\documentclass[11pt]{amsart}
\usepackage[margin=1in]{geometry}
\usepackage{latexsym}
\usepackage{amsfonts}
\usepackage{amsmath}
\usepackage{amssymb}
\usepackage{amsthm}
\usepackage{enumerate}
\usepackage[hang,flushmargin]{footmisc}
\usepackage{caption}
\usepackage{tabu}
\usepackage{mathrsfs}
\usepackage{amsaddr}
\usepackage{subfig}

\usepackage{graphicx}
\usepackage{epstopdf}
\usepackage{epsfig}
\usepackage{caption}
 
\usepackage{bm} 
\usepackage{cite}

\usepackage{fancyhdr}
\makeatletter
\newtheorem*{rep@theorem}{\rep@title}
\newcommand{\newreptheorem}[2]{%
\newenvironment{rep#1}[1]{%
 \def\rep@title{#2 \ref{##1}}%
 \begin{rep@theorem}}%
 {\end{rep@theorem}}}
\makeatother

\newtheorem{theorem}{Theorem}[section]
\newreptheorem{theorem}{Conjecture}

\newtheorem{proposition}{Proposition}[section]
\newtheorem{corollary}{Corollary}[section]

\theoremstyle{definition}
\newtheorem{definition}{Definition}[section]

\DeclareMathOperator{\des}{des}

\begin{document}
\title{Descents in $t$-Sorted Permutations}
\author{Colin Defant}
\address{Princeton University \\ Fine Hall, 304 Washington Rd. \\ Princeton, NJ 08544}
\email{cdefant@princeton.edu}

\begin{abstract}
Let $s$ denote West's stack-sorting map. A permutation is called \emph{$t$-sorted} if it is of the form $s^t(\mu)$ for some permutation $\mu$. We prove that the maximum number of descents that a $t$-sorted permutation of length $n$ can have is $\left\lfloor\frac{n-t}{2}\right\rfloor$. When $n$ and $t$ have the same parity and $t\geq 2$, we give a simple characterization of those $t$-sorted permutations in $S_n$ that attain this maximum. In particular, the number of such permutations is $(n-t-1)!!$. 
\end{abstract}

\keywords{permutation; descent; stack-sorting; valid hook configuration}

\maketitle

\section{Introduction}

In this paper, a ``permutation" is a permutation of a finite set of positive integers, written in one-line notation. Let $S_n$ denote the set of all permutations of the set $[n]$. In his Ph.D. dissertation, West \cite{West} introduced a function $s$, called the \emph{stack-sorting map}, that sends permutations through a vertical ``stack" as follows. Suppose we are given an input permutation $\pi=\pi_1\cdots\pi_n$. At any point in time during the procedure, if the next entry in the input permutation is smaller than the entry at the top of the stack or if the stack is empty, the next entry in the input permutation is placed at the top of the stack. Otherwise, the entry at the top of the stack is appended to the end of the growing output permutation. This process terminates when the output permutation has length $n$, and $s(\pi)$ is defined to be this output permutation. The following illustration shows that $s(4162)=1426$. 

\begin{center}
\includegraphics[width=1\linewidth]{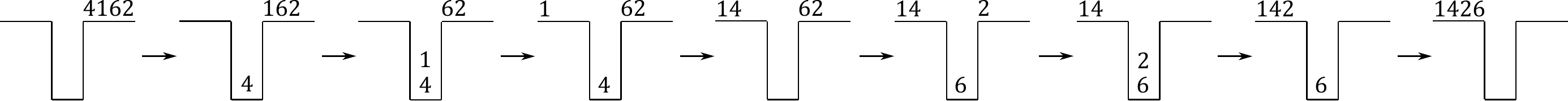}
\end{center}

West's stack-sorting map was actually defined as a deterministic variant of a ``stack-sorting algorithm" that Knuth introduced in \cite{Knuth}. In fact, Knuth's analysis of his stack-sorting algorithm initiated the investigation of permutation patterns, which is now a major area of research \cite{Bona, Kitaev}. It was also the first appearance of the so-called ``kernel method," which is now an indispensable tool in enumerative and analytic combinatorics \cite{Banderier, Bousquet3}.  The stack-sorting map has received a huge amount of attention since its introduction in West's dissertation \cite{Bona, BonaSurvey, BonaSimplicial, BonaSymmetry, BonaWords, Bousquet98, Bousquet, Bouvel, BrandenActions, Branden3, Claesson, Cori, DefantCatalan, DefantCounting, DefantEnumeration, DefantFertility, DefantFertilityWilf, DefantPolyurethane, DefantPostorder, DefantPreimages, DefantClass, DefantEngenMiller, DefantKravitz, Dulucq, Dulucq2, Fang, Goulden, Ulfarsson, West, Zeilberger}. We will mention just a few results in this line of work, referring the reader to \cite{Bona, BonaSurvey, DefantCatalan, DefantCounting, DefantEnumeration, DefantFertility, DefantFertilityWilf, DefantPolyurethane, DefantPostorder, DefantPreimages, DefantClass, DefantEngenMiller, DefantKravitz} for more details. 

Bousquet-M\'elou defined a permutation to be \emph{sorted} if it is in the image of $s$, and she described a method that allows one to determine whether or not a given permutation is sorted. She also found a bivariate generating function equation that implicitly enumerates sorted permutations, but she was unable to remove the additional ``catalytic variable." In short, this means that counting sorted permutations explicitly (or even obtaining asymptotic information) is hard. 

In recent years, the current author \cite{DefantCatalan, DefantCounting, DefantEnumeration, DefantFertility, DefantFertilityWilf, DefantPostorder, DefantPreimages, DefantClass, DefantEngenMiller, DefantKravitz} has introduced objects called ``valid hook configurations" in order to reprove and generalize old results and to prove new results concerning the map $s$. These objects allow one to compute the \emph{fertility} of a permutation, which is the number of preimages of the permutation under $s$. In particular, they give a method, which we describe in Section \ref{Sec:VHC}, for determining if a permutation is sorted. This method and Bousquet-M\'elou's have some similarities, but we believe the former is better suited for our purposes. 

A \emph{descent} of a permutation $\pi=\pi_1\cdots\pi_n$ is an index $i\in[n-1]$ such that $\pi_i>\pi_{i+1}$. Let $\des(\pi)$ denote the number of descents of $\pi$. It is known (see either \cite{DefantEngenMiller} or Exercise 18 in Chapter 8 of \cite{Bona}) that every sorted permutation of length $n$ has at most $\frac{n-1}{2}$ descents. The authors of \cite{DefantEngenMiller} studied the permutations that attain this maximum, which turn out to have several interesting properties. The exploration of these permutations began with the following characterization. Let us say a permutation is \emph{uniquely sorted} if it has exactly one preimage under $s$. 

\begin{theorem}[\!\!\cite{DefantEngenMiller}]\label{Thm1}
A permutation of length $n$ is uniquely sorted if and only if it is sorted and has exactly $\frac{n-1}{2}$ descents. 
\end{theorem}

The previous theorem implies that every uniquely sorted permutation has odd length. The authors of \cite{DefantEngenMiller} defined a bijection between uniquely sorted permutations and certain weighted matchings that Josuat-Verg\`es \cite{Josuat} studied in the context of free probability theory. From this, they deduced that the number of uniquely sorted permutations in $S_{2k+1}$ is $A_{k+1}$, where $(A_m)_{m\geq 1}$ is sequence A180874 in the OEIS and is known as \emph{Lassalle's sequence} \cite{OEIS}. This exciting new sequence first appeared in \cite{Lassalle}, where Lassalle proved a conjecture of Zeilberger by showing that it is increasing. In fact, the bijection established in \cite{DefantEngenMiller} produced three new combinatorial interpretations of Lassalle's sequence; the only combinatorial interpretation known beforehand involved the weighted matchings that Josuat-Verg\`es examined. The authors of \cite{DefantEngenMiller} also showed that the sequences $(A_{k+1}(\ell))_{\ell=1}^{2k+1}$ are symmetric, where $A_{k+1}(\ell)$ is the number of uniquely sorted permutations in $S_{2k+1}$ that start with the number $\ell$. One can define the \emph{hotspot} of a uniquely sorted permutation $\pi_1\cdots\pi_n$ to be $\pi_{r+1}$, where $r$ is the largest element of $[n-1]$ such that $\pi$ has $\frac{n-r}{2}$ descents in $\{r,\ldots,n-1\}$. This somewhat strange definition is justified by the surprising fact that $A_{k+1}(\ell)$ is the number of uniquely sorted permutations in $S_{2k+1}$ with hotspot $\ell-1$. More recently, the current author \cite{DefantEnumeration} has found several bijections between sets of uniquely sorted permutations avoiding various patterns and intervals in posets of Dyck paths.  

It is typical to think of the stack-sorting map $s$ as producing a dynamical system on $S_n$. Thus, we let $s^t$ denote the composition of $s$ with itself $t$ times. It is straightforward to check that $s^{n-1}(\pi)=123\cdots n$ for every $\pi\in S_n$. Consequently, we can endow $S_n$ with the structure of a rooted tree (the ``stack-sorting tree on $S_n$") by letting $123\cdots n$ be the root and declaring that a nonidentity permutation $\sigma$ is a child of $\pi$ if $s(\sigma)=\pi$ (see Figure \ref{Fig1}). One of the most well-studied notions concerning the stack-sorting map is that of a \emph{$t$-stack-sortable} permutation \cite{Bona, BonaSimplicial, BonaSurvey, BonaSymmetry, BonaWords, Bousquet98, Bousquet, BrandenActions, Branden3, Cori, DefantCounting, DefantPreimages, Dulucq, Dulucq2, Fang, Goulden, West, Zeilberger}, which is a permutation $\pi$ such that $s^t(\pi)$ is increasing. When we restrict attention to $S_n$, we see that these are the permutations of depth at most $t$ in the stack-sorting tree on $S_n$. The definition of a sorted permutation is in some sense dual to that of a $1$-stack-sortable permutation. Indeed, a permutation in $S_n$ is sorted if and only if it has height at least $1$ in the stack-sorting tree on $S_n$. In this article, we consider permutations of height at least $t$ in this stack-sorting tree. This naturally generalizes the definition of a sorted permutation, providing a dual to the notion of a $t$-stack-sortable permutation.  

\begin{figure}[h]
\begin{center}
\includegraphics[width=.7\linewidth]{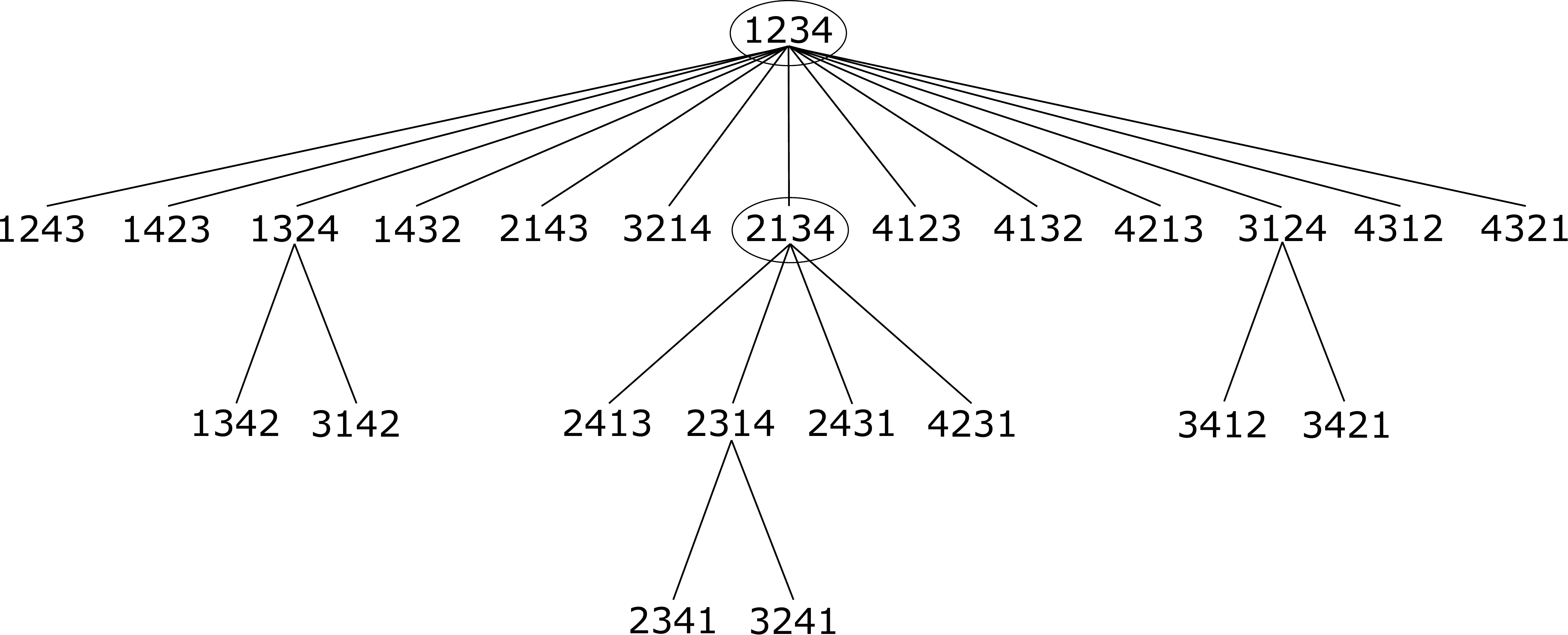}
\caption{The stack-sorting tree on $S_4$. The only $2$-sorted permutations in $S_4$ are $1234$ and $2134$, which are circled.} 
\label{Fig1}
\end{center}  
\end{figure}

\begin{definition}\label{Def1}
A permutation is called \emph{$t$-sorted} if it is of the form $s^t(\mu)$ for some permutation $\mu$. 
\end{definition}

Our main results are as follows. We phrase these results for permutations in $S_n$, but the analogous statements for permutations of arbitrary finite sets of positive integers hold as well. Recall that a \emph{left-to-right maximum} of a permutation is an entry that is larger than everything to its left. 

\begin{theorem}\label{Thm2}
If $n\geq t\geq 1$, then the maximum number of descents that a $t$-sorted permutation in $S_n$ can have is $\left\lfloor\frac{n-t}{2}\right\rfloor$. 
\end{theorem}
                 
\begin{theorem}\label{Thm3}
Suppose that $n\geq t\geq 2$ and that $n\equiv t\pmod 2$. A permutation $\pi=\pi_1\cdots\pi_n\in S_n$ is $t$-sorted and has $\frac{n-t}{2}$ descents if and only if its left-to-right maxima are \[\pi_1,\pi_3,\pi_5,\ldots,\pi_{n-t+1},\pi_{n-t+2},\pi_{n-t+3},\ldots,\pi_n.\] In particular, the number of such permutations is $(n-t-1)!!$.  
\end{theorem}

The quantity $\max\{\des(s^t(\mu)):\mu\in S_n\}$ drops by roughly a factor of $2$ when $t$ changes from $0$ to $1$. One might expect this quantity to drop by another constant factor when $t$ changes from $1$ to $2$ or from $2$ to $3$. However, Theorem \ref{Thm2} tells us that this is not actually the case; when $t\geq 1$ and we increment $t$ by $1$, this maximum decreases by at most $1$. 

A general rule of thumb for dynamical systems is that things get much more complicated as one considers higher and higher iterates. This is certainly true in the context of $t$-stack-sortable permutations. It follows from Knuth's analysis \cite{Knuth} that a permutation is $1$-stack-sortable if and only if it avoids the pattern $231$, so the number of such permutations in $S_n$ is simply the $n^\text{th}$ Catalan number $C_n=\frac{1}{n+1}{2n\choose n}$. West \cite{West} gave a more complicated characterization of $2$-stack-sortable permutations and conjectured that the number of such permutations in $S_n$ is $\frac{2}{(n+1)(2n+1)}{3n\choose n}$. This was proven by Zeilberger \cite{Zeilberger}, and other proofs emerged later \cite{Cori,DefantCounting,Dulucq,Dulucq2,Goulden}. There is a much more complicated characterization of $3$-stack-sortable permutations due to \'Ulfarsson involving so-called ``decorated patterns," and only very recently has a (very complicated) recurrence for these numbers emerged \cite{DefantCounting}. Morally speaking, the article \cite{DefantCounting} tells us that $3$-stack-sortable permutations fail to conform to some of the nice patterns that $1$-stack-sortable permutations and $2$-stack-sortable permutations obey. We do not even have a characterization of $4$-stack-sortable permutations. 

In light of this rule of thumb for dynamical systems, the utter simplicity of the characterization in Theorem \ref{Thm3} is shocking. This theorem tells us that the set of extremal permutations attaining the maximum number of descents actually becomes much simpler when we consider $t$-sorted permutations for $t\geq 2$ instead of sorted permutations. Indeed, recall from Theorem \ref{Thm1} that the sorted permutations in $S_n$ with exactly $\frac{n-1}{2}$ descents are precisely the uniquely sorted permutations in $S_n$. These permutations are counted by Lassalle's sequence, which is quite complicated (and intriguing!). This sequence did not even appear in the literature until 2012. By contrast, when $t\geq 2$, the extremal permutations are counted by double factorials, which were understood well before 2012. It is also interesting that when $t\geq 2$, the number of such permutations only depends on the difference $n-t$, which is also twice the number of descents in these permutations. 

For emphasis, let us reiterate that the analogue of the characterization in Theorem \ref{Thm3} for $t=1$ is false. One direction is true. If $n$ is odd and $\pi=\pi_1\cdots\pi_n\in S_n$ has $\pi_1,\pi_3,\pi_5,\ldots,\pi_n$ as its left-to-right maxima, then $\pi$ is sorted and has $\frac{n-1}{2}$ descents (equivalently, it is uniquely sorted). However, when $n\geq 5$ is odd, there are uniquely sorted permutations of length $n$ whose left-to-right maxima are not the entries in odd-indexed positions. For example, the uniquely sorted permutations in $S_5$ are $21435, 31425, 32145, 32415, 42135$. 

\section{Valid Hook Configurations}\label{Sec:VHC}
We now define valid hook configurations and state how to use them to determine if a permutation is sorted. We only need valid hook configurations in order to prove Corollary \ref{Cor1} below, so the reader wishing to skip this discussion can simply accept Corollary \ref{Cor1} on the basis of faith and proceed to Section \ref{Sec:Main}. 

The \emph{plot} of a permutation $\pi=\pi_1\cdots\pi_n$ is the graph displaying the points $(i,\pi_i)$ for all $i\in[n]$. The left image in Figure \ref{Fig2} shows the plot of $3142567$. A \emph{hook} of $\pi$ is drawn by starting at a point $(i,\pi_i)$ in the plot of $\pi$, moving vertically upward, and then moving to the right until reaching another point $(j,\pi_j)$. In order for this to make sense, we must have $i<j$ and $\pi_i<\pi_j$. The point $(i,\pi_i)$ is called the \emph{southwest endpoint} of the hook, while $(j,\pi_j)$ is called the \emph{northeast endpoint}. The right image in Figure \ref{Fig2} shows the plot of $3142567$ along with a hook that has southwest endpoint $(3,4)$ and northeast endpoint $(6,6)$.

\begin{figure}[h]
  \centering
  \subfloat[]{\includegraphics[width=0.185\textwidth]{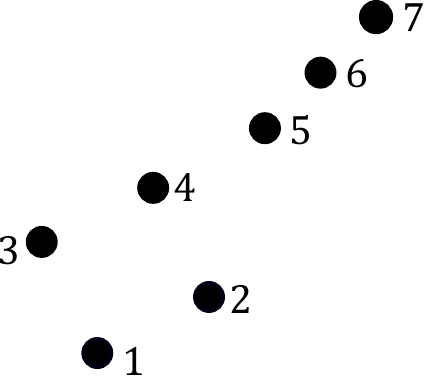}}
  \hspace{1.5cm}
  \subfloat[]{\includegraphics[width=0.185\textwidth]{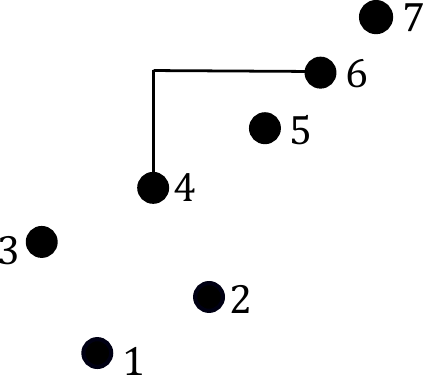}}
  \caption{The left image is the plot of $3142567$. The right image shows this plot along with a single hook.}\label{Fig2}
\end{figure}

\begin{definition}\label{Def2}
Let $\pi=\pi_1\cdots\pi_n$ be a permutation whose descents are $d_1<\cdots<d_k$. Let $\mathcal H=(H_1,\ldots,H_k)$ be a tuple of hooks of $\pi$. Let $(i_u,\pi_{i_u})$ and $(j_u,\pi_{j_u})$ be the southwest endpoint and the northeast endpoint of $H_u$, respectively. We say $\mathcal H$ is a \emph{valid hook configuration} of $\pi$ if the following conditions are satisfied: 

\begin{enumerate}[1.]
\item We have $i_u=d_u$ for every $u\in\{1,\ldots,k\}$.

\item No point in the plot of $\pi$ lies directly above a hook in $\mathcal H$. 

\item The hooks in $\mathcal H$ do not intersect each other except in the case that the northeast endpoint of one hook is the southwest endpoint of another. 
\end{enumerate}  
\end{definition}

\begin{figure}[t]
\begin{center}
\includegraphics[width=.65\linewidth]{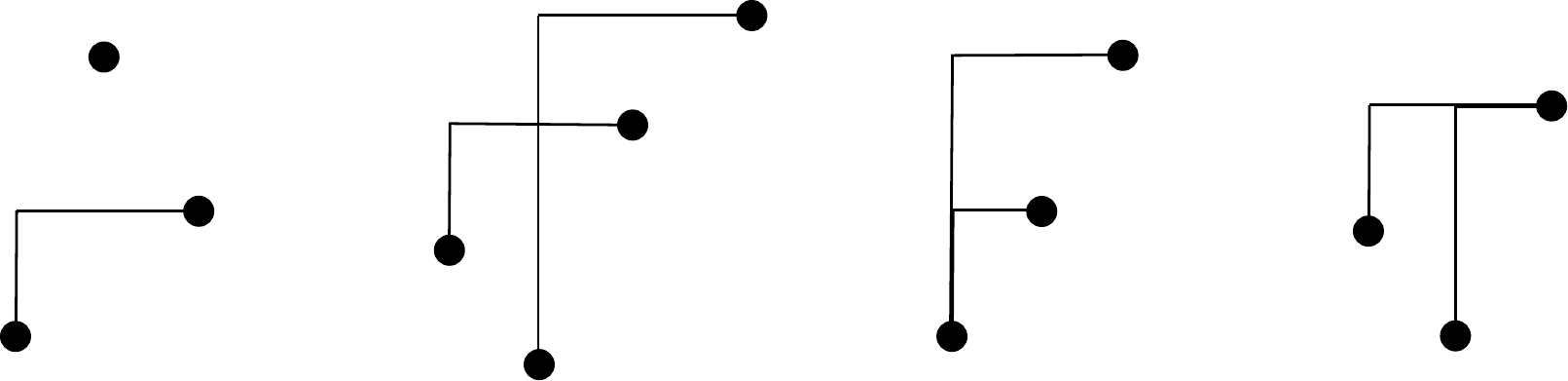}
\caption{Four arrangements of hooks that are forbidden in a valid hook configuration.}
\label{Fig3}
\end{center}  
\end{figure}

\begin{figure}[t]
\begin{center}
\includegraphics[width=.65\linewidth]{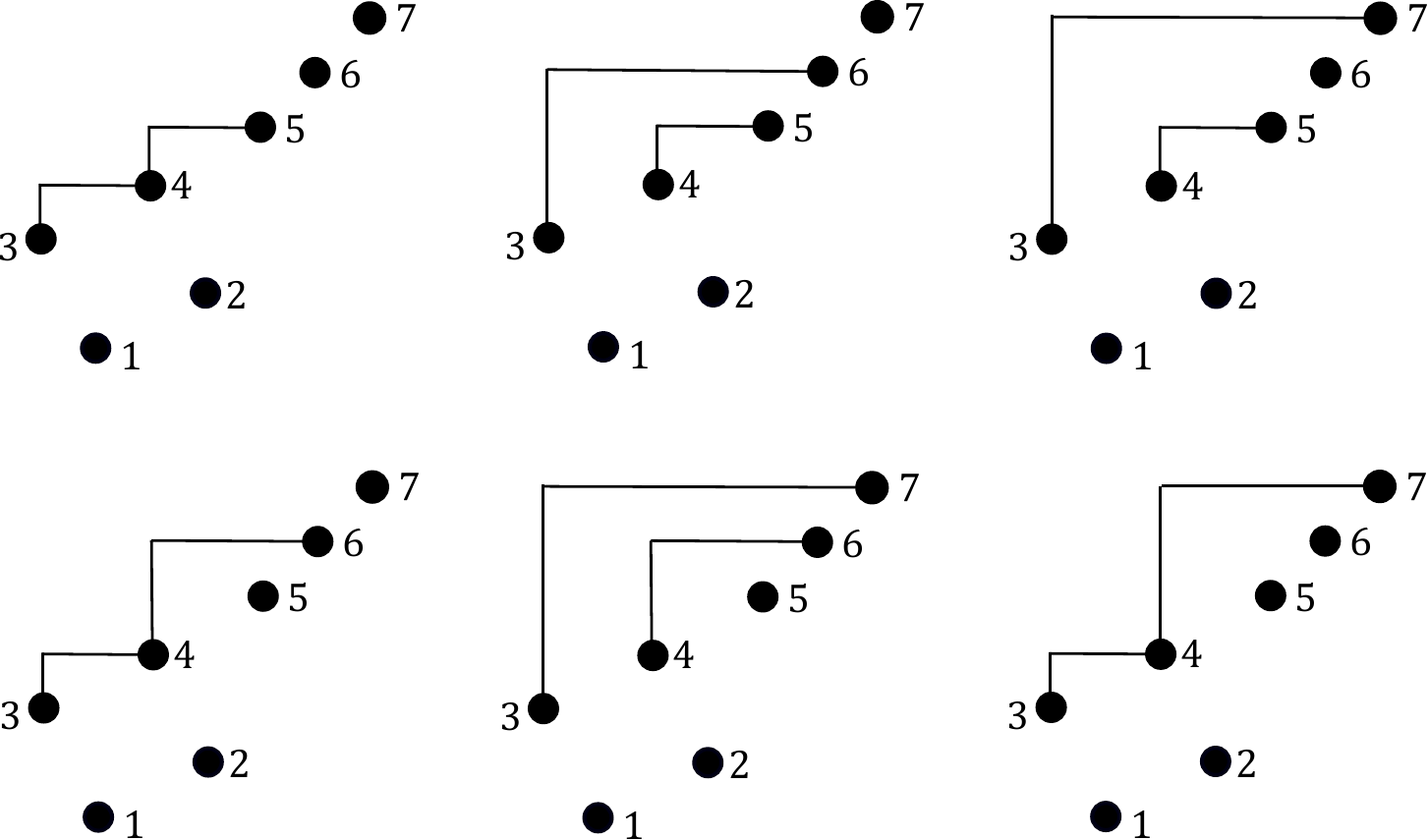}
\caption{The valid hook configurations of $3142567$.}
\label{Fig4}
\end{center}  
\end{figure}

Figure \ref{Fig3} shows arrangements of hooks that are forbidden from appearing in a valid hook configuration by Conditions 2 and 3 in Definition \ref{Def2}.
Figure \ref{Fig4} shows all of the valid hook configurations of $3142567$. 
Observe that the total number of hooks in a valid hook configuration of $\pi$ is exactly $k$, the number of descents of $\pi$. 

The following theorem tells us how to use valid hook configurations to determine whether or not a given permutation is sorted. It is a special consequence of Theorem 5.1 in \cite{DefantPostorder}.\footnote{The article \cite{DefantPostorder} from 2017 is slightly outdated. We refer the reader to \cite{DefantClass} (specifically, Theorem 2.1 in that article) for a more modern treatment of valid hook configurations.} 

\begin{theorem}[\!\!\cite{DefantPostorder}]\label{Thm4}
A permutation is sorted if and only if it has a valid hook configuration. 
\end{theorem}

We can now combine Theorems \ref{Thm1} and \ref{Thm4} to prove the following corollary. This proof is the only place where we explicitly use valid hook configurations. However, we will continue to rely heavily on Theorem \ref{Thm1}, whose proof also uses valid hook configurations. 

\begin{corollary}\label{Cor1}
Let $\pi=\pi_1\cdots\pi_n$ be a permutation. If there is an index $\ell\in[n-2]$ such that $\pi_{\ell+1}<\pi_{\ell+2}<\pi_\ell$, then $\pi$ is not uniquely sorted. 
\end{corollary} 

\begin{proof}
Suppose instead that such an index $\ell$ exists and that $\pi$ is uniquely sorted. Let $k=\des(\pi)$. According to Theorem \ref{Thm1}, $n=2k+1$. Because $\pi$ is sorted, Theorem \ref{Thm4} tells us that it has a valid hook configuration $\mathcal H$. It follows from Condition 3 in Definition \ref{Def2} that the northeast endpoints of the hooks in $\mathcal H$ are distinct. There are $k$ hooks in $\mathcal H$, so there are $k$ northeast endpoints of hooks in $\mathcal H$. Let us say a point $(i,\pi_i)$ in the plot of $\pi$ is a \emph{descent bottom} of the plot of $\pi$ if $i-1$ is a descent of $\pi$. Note that the plot of $\pi$ has exactly $k$ descent bottoms. It follows from Condition 2 in Definition \ref{Def2} that a descent bottom cannot be the northeast endpoint of a hook in $\mathcal H$. Since $n=2k+1$, this implies that the set of descent bottoms of the plot of $\pi$ and the set of northeast endpoints of hooks in $\mathcal H$ form a partition of $\{(i,\pi_i):2\leq i\leq n\}$ into two sets of size $k$. 

Now consider the point $(\ell+2,\pi_{\ell+2})$. This point is not a descent bottom of the plot of $\pi$, so it follows from the previous paragraph that it is the northeast endpoint of a hook $H$ in $\mathcal H$. According to Condition 2 in Definition \ref{Def2}, $H$ cannot pass below the point $(\ell,\pi_\ell)$. This means that the southwest endpoint of $H$ must be $(\ell+1,\pi_{\ell+1})$. However, this contradicts Condition 1 in Definition \ref{Def2} because $\ell+1$ is not a descent of $\pi$.     
\end{proof}

\section{Proofs of Main Results}\label{Sec:Main}
The purpose of this section is to prove Theorems \ref{Thm2} and \ref{Thm3}. Theorem \ref{Thm2} is already known when $t=1$, and Theorem \ref{Thm3} is only stated for $t\geq 2$. Therefore, we may assume $n\geq t\geq 2$. 

\begin{proof}[Proof of Theorem \ref{Thm2}]

Let us begin by proving that every $t$-sorted permutation in $S_n$ has at most $\frac{n-t}{2}$ descents. Observe that a $t$-sorted permutation in $S_n$ must end in the entries $n-t+1,n-t+2,\ldots,n$, in that order. This is a consequence of the definition of the stack-sorting map, and it is the reason why $s^{n-1}(\pi)=123\cdots n$ for every $\pi\in S_n$. Now suppose $\pi\in S_n$ is $t$-sorted. We can write $\pi=\pi'(n-t+2)(n-t+3)\cdots n$ for some permutation $\pi'\in S_{n-t+1}$ that ends in the entry $n-t+1$. There is a $(t-1)$-sorted permutation $\sigma$ such that $s(\sigma)=\pi$. Because $\sigma$ is $(t-1)$-sorted, we have $\sigma=\sigma'(n-t+2)(n-t+3)\cdots n$ for some $\sigma'\in S_{n-t+1}$. By applying the stack-sorting procedure to $\sigma$, we find that $\pi=s(\sigma'(n-t+2)(n-t+3)\cdots n)=s(\sigma')(n-t+2)(n-t+3)\cdots n$. Thus, $\pi'=s(\sigma')$. This means that $\pi'$ is a sorted permutation in $S_{n-t+1}$, so it has at most $\frac{n-t}{2}$ descents. Hence, $\pi$ also has at most $\frac{n-t}{2}$ descents. 

For our next point of business, we assume $n\equiv t\pmod 2$ and consider a permutation $\pi=\pi_1\cdots\pi_n\in S_n$ whose left-to-right maxima are $\pi_1,\pi_3,\pi_5,\ldots,\pi_{n-t+1},\pi_{n-t+2},\pi_{n-t+3},\ldots,\pi_n$. This permutation has exactly $\frac{n-t}{2}$ descents, which are precisely the elements of $\{1,3,5,\ldots,n-t-1\}$. We wish to show that $\pi$ is $t$-sorted. Let $\pi^{(0)}=\pi$. Let $\pi^{(1)}$ be the permutation obtained from $\pi^{(0)}$ by sliding each of the entries $\pi_2,\pi_4,\pi_6,\ldots,\pi_{n-t}$ to the right by $1$ position. For example, if $t=3$, $n=11$, and $\pi=\pi^{(0)}=5\,1\,6\,2\,7\,3\,8\,4\,9\,10\,11$, then $\pi^{(1)}=5\,6\,1\,7\,2\,8\,3\,9\,4\,10\,11$. Now let $\pi^{(2)}$ be the permutation obtained from $\pi^{(1)}$ by sliding each of the entries $\pi_2,\pi_4,\pi_6,\ldots,\pi_{n-t}$ (the same entries as before) to the right by $1$ position. In the above example, $\pi^{(2)}=5\,6\,7\,1\,8\,2\,9\,3\,10\,4\,11$. Continue in this fashion to construct the permutations $\pi^{(1)},\pi^{(2)},\ldots,\pi^{(t)}$. In the above example, $\pi^{(t)}=\pi^{(3)}=5\,6\,7\,8\,1\,9\,2\,10\,3\,11\,4$. It is straightforward to check that $s(\pi^{(i)})=\pi^{(i-1)}$ for every $i\in\{1,\ldots,t\}$. This shows that $\pi=s^t(\pi^{(t)})$, so $\pi$ is $t$-sorted. 

The previous two paragraphs prove Theorem \ref{Thm2} when $n$ and $t$ have the same parity. In order to complete the proof, we need to show that there is a $t$-sorted permutation in $S_n$ with $\frac{n-t-1}{2}$ descents when $n\not\equiv t\pmod 2$. In this case, we know already that there is a $t$-sorted permutation $\lambda\in S_{n-1}$ with $\frac{n-t-1}{2}$ descents. Let $\mu\in S_{n-1}$ be such that $s^t(\mu)=\lambda$. Let $1\oplus\lambda\in S_n$ be the permutation obtained by incrementing each entry in $\lambda$ by $1$ and then prepending a $1$ to the resulting permutation. For example, if $\lambda=324156$, then $1\oplus\lambda=1435267$. Define $1\oplus\mu\in S_n$ similarly. It is straightforward to check that $s^t(1\oplus \mu)=1\oplus\lambda$, so $1\oplus\lambda$ is a $t$-sorted permutation in $S_n$ with $\frac{n-t-1}{2}$ descents. 
\end{proof}

Having completed the proof of Theorem \ref{Thm2}, we proceed to prove Theorem \ref{Thm3}. Let us first prove the characterization stated in this theorem in the case in which $t=2$. We are given that $n\geq 2$ is even. We saw above that every permutation $\pi=\pi_1\cdots\pi_n\in S_n$ whose left-to-right maxima are $\pi_1,\pi_3,\pi_5,\ldots,\pi_{n-1},\pi_n$ is $2$-sorted and has $\frac{n-2}{2}$ descents. We need to prove the converse, which is the statement of the following proposition. 

\begin{proposition}\label{Prop1}
Let $n\geq 2$ be even. Let $\pi=\pi_1\cdots\pi_n\in S_n$ be a $2$-sorted permutation with $\frac{n-2}{2}$ descents. The left-to-right maxima of $\pi$ are $\pi_1,\pi_3,\pi_5,\ldots,\pi_{n-1},\pi_n$. 
\end{proposition}

\begin{proof}
Because $\pi$ is $2$-sorted, there are permutations $\sigma=\sigma_1\cdots\sigma_n$ and $\mu=\mu_1\cdots\mu_n$ in $S_n$ such that $s(\mu)=\sigma$ and $s(\sigma)=\pi$. Since $\pi$ and $\sigma$ are sorted, we have $\pi_n=\sigma_n=n$. Thus, we can write $\pi=\pi'n$ and $\sigma=\sigma'n$, where $\pi'=\pi_1\cdots\pi_{n-1}$ and $\sigma'=\sigma_1\cdots\sigma_{n-1}$. We have $\pi=s(\sigma)=s(\sigma'n)=s(\sigma')n$, so $s(\sigma')=\pi'$. This shows that $\pi'$ is a sorted permutation in $S_{n-1}$ with $\frac{n-2}{2}$ descents, so it is uniquely sorted by Theorem \ref{Thm1}. We now prove a sequence of claims regarding the permutation $\sigma$. Recall that a point $(u,\lambda_u)$ in the plot of a permutation $\lambda=\lambda_1\cdots\lambda_n$ is called a descent bottom of the plot of $\lambda$ if $u-1$ is a descent of $\lambda$. In this case, we also say that the \emph{entry} $\lambda_u$ is a descent bottom of $\lambda$. For example, the descent bottoms of $5346127$ are $1$ and $3$. We say an index $i\in\{2,\ldots,n-1\}$ is a \emph{double descent} of $\lambda$ if $\lambda_{i-1}>\lambda_i>\lambda_{i+1}$. 

\noindent {\bf Claim 1}: The permutation $\sigma$ has no double descents. 

From the definition of the stack-sorting map, it is straightforward to verify that every descent bottom of the permutation $\pi=s(\sigma)$ is also a descent bottom of $\sigma$. This implies that $\des(\sigma)\geq\des(\pi)=\frac{n-2}{2}$. Now, $\sigma$ is a sorted permutation in $S_n$, so $\des(\sigma)\leq\left\lfloor\frac{n-1}{2}\right\rfloor=\frac{n-2}{2}$. It follows that $\des(\sigma)=\des(\pi)=\frac{n-2}{2}$. Since every descent bottom of $\pi$ is a descent bottom of $\sigma$, we now know that every descent bottom of $\sigma$ is a descent bottom of $\pi$. Suppose $i$ is a double descent of $\sigma$. When we apply the stack-sorting procedure to $\sigma$, there will be a point in time when the entry $\sigma_{i+1}$ sits on top of $\sigma_i$ in the stack. Whichever entry leaves the stack immediately before $\sigma_i$ leaves the stack must be smaller than $\sigma_i$ since it must have sat on top of $\sigma_i$ in the stack. This prohibits $\sigma_i$ from being a descent bottom of $\pi$, which is a contradiction since it is a descent bottom of $\sigma$. This proves Claim 1. 

\noindent {\bf Claim 2}: We have $\sigma_1<\sigma_2$. 

Suppose instead that $\sigma_1>\sigma_2$. Since $\sigma$ has no double descents by Claim 1, we must have $\sigma_2<\sigma_3$. Let $\sigma''=\sigma_2\sigma_1\sigma_3\sigma_4\cdots\sigma_{n-1}\in S_{n-1}$ be the permutation
obtained from $\sigma'$ by switching the positions of its first two entries. Consider sending $\sigma'$ and $\sigma''$ through different stacks simultaneously. When stack-sorting $\sigma'$, the first step is to push $\sigma_1$ into the stack. Next, we push $\sigma_2$ into the stack. The third step is to pop $\sigma_2$ out of the stack (because $\sigma_2<\sigma_3$). When stack-sorting $\sigma''$, the first step is to push $\sigma_2$ into the stack. Next, we pop $\sigma_2$ out of the stack (because $\sigma_2<\sigma_1$). The third step is to push $\sigma_1$ into the stack. Thus, after taking three steps each, the two stack-sorting procedures are in identical configurations. Indeed, in both procedures, $\sigma_2$ is the only entry that has left the stack, $\sigma_1$ is the only entry in the stack, and $\sigma_3\sigma_4\cdots\sigma_{n-1}$ is the remainder of the input permutation consisting of those entries that have not yet entered the stack. From this, it follows that $s(\sigma')=s(\sigma'')$. However, this is a contradiction because $s(\sigma')$ is $\pi'$, which we previously showed is uniquely sorted. This proves Claim 2. 

\noindent {\bf Claim 3}: There is a permutation $\lambda=\lambda_1\cdots\lambda_n\in S_n$ such that $s(\lambda)=\sigma$ and $\lambda_1=\sigma_1$. 

The permutation $\mu\in S_n$ satisfies $s(\mu)=\sigma$, so we are done if $\sigma_1=\mu_1$. Thus, we may assume $\sigma_1=\mu_j$ for some $j\in\{2,\ldots,n\}$. When we send $\mu$ through the stack-sorting procedure, $\mu_j$ is the first entry to leave the stack (because it is $\sigma_1$). This forces $\mu_1>\cdots>\mu_j$ and $\mu_{j+1}>\mu_j$. Let $\lambda=\lambda_1\cdots\lambda_n=\mu_j\mu_1\mu_2\cdots\mu_{j-1}\mu_{j+1}\mu_{j+2}\cdots\mu_n$ be the permutation obtained from $\mu$ by moving the entry $\mu_j$ to the beginning of the permutation and keeping all other entries in the same relative order. Consider sending $\mu$ and $\lambda$ through different stacks simultaneously. When stack-sorting $\mu$, the first step is to push $\mu_1$ into the stack. The second step is to push $\mu_2$ into the stack. We continue until pushing $\mu_j$ into the stack in the $j^\text{th}$ step. The $(j+1)^\text{st}$ step is to pop $\mu_j$ out of the stack (because $\mu_j<\mu_{j+1}$). When stack-sorting $\lambda$, the first step is to push $\mu_j$ into the stack. The second step is to pop $\mu_j$ out of the stack (because $\mu_j<\mu_1$). The third step is to push $\mu_1$ into the stack. We then continue until pushing $\mu_{j-1}$ into the stack in the $(j+1)^\text{st}$ step. Thus, after taking $j+1$ steps each, the two sorting procedures are in identical configurations. Indeed, in both procedures, $\mu_j$ is the only entry that has left the stack, $\mu_1,\mu_2,\ldots,\mu_{j-1}$ are the entries in the stack (listed here from bottom to top), and $\mu_{j+1}\mu_{j+2}\cdots\mu_n$ is the remainder of the input permutation consisting of those entries that have not yet entered the stack. This shows that $s(\lambda)=s(\mu)=\sigma$, so this choice of $\lambda$ has the properties needed to prove Claim 3. 

\noindent {\bf Claim 4}: The permutation $\zeta=\sigma_2\sigma_3\cdots\sigma_n$ is uniquely sorted. 

Let $\lambda$ be the permutation that is guaranteed to exist by Claim 3. When we send $\lambda$ through the stack-sorting procedure, $\lambda_1$ is the first entry to leave the stack (because it is $\sigma_1$). This means that nothing can ever sit on top of $\lambda_1$ in the stack, so $\lambda_1<\lambda_2$. It follows that $\sigma_1\zeta=\sigma=s(\lambda)=\lambda_1s(\lambda_2\lambda_3\cdots\lambda_n)=\sigma_1s(\lambda_2\lambda_3\cdots\lambda_n)$, so $\zeta$ is sorted. We saw in the proof of Claim 1 that $\des(\sigma)=\frac{n-2}{2}$. Claim 2 tells us that $1$ is not a descent of $\sigma$, so $\zeta$ must have $\frac{n-2}{2}$ descents. Since $\zeta$ is a sorted permutation of length $n-1$, it follows from Theorem \ref{Thm1} that $\zeta$ is uniquely sorted. 

\noindent {\bf Claim 5}: For every descent $i$ of $\sigma$, we have $\sigma_i<\sigma_{i+2}$. 

Suppose $i$ is a descent of $\sigma$. Note that $i\neq n-1$ since $\sigma_n=n$ ($\sigma$ is sorted). This means that it makes sense to talk about the entry $\sigma_{i+2}$. Claim 2 tells us that $i\geq 2$, so $i$ is a descent of the permutation $\zeta=\sigma_2\sigma_3\cdots\sigma_n$. Claim 4 tells us that $\zeta$ is uniquely sorted, so it follows from Corollary \ref{Cor1} that we do not have $\sigma_{i+1}<\sigma_{i+2}<\sigma_i$. Claim 1 tells us that $i+1$ is not a double descent of $\sigma$, so we do not have $\sigma_{i+2}<\sigma_{i+1}<\sigma_i$. The only remaining possibility is that $\sigma_{i+1}<\sigma_i<\sigma_{i+2}$. 

\noindent {\bf Claim 6}: The descents of $\sigma$ are $2,4,6,\ldots,n-2$, and the left-to-right maxima of $\sigma$ are $\sigma_1,\sigma_2,\sigma_4,\sigma_6,$ $\ldots,\sigma_n$.   

We saw in the proof of Claim 1 that $\sigma$ has $\frac{n-2}{2}$ descents, and Claim 1 itself guarantees that no two of these descents are consecutive integers. We also know by Claim 2 that $1$ is not a descent of $\sigma$. Furthermore, since $\sigma_n=n$ ($\sigma$ is sorted), the index $n-1$ is not a descent of $\sigma$. Put together, these facts force the descents of $\sigma$ to be $2,4,6,\ldots,n-2$. Now, $\sigma_1$ is obviously a left-to-right maximum of $\sigma$. We also know that $\sigma_2$ is a left-to-right maximum by Claim 2. Since $2$ is a descent of $\sigma$, we know by Claim 5 that $\sigma_3<\sigma_2<\sigma_4$. This shows that $\sigma_3$ is not a left-to-right maximum and that $\sigma_4$ is a left-to-right maximum. Since $4$ is a descent of $\sigma$, we know by Claim 5 that $\sigma_5<\sigma_4<\sigma_6$. This shows that $\sigma_5$ is not a left-to-right maximum and that $\sigma_6$ is. Continuing in this fashion, we find that the left-to-right maxima of $\sigma$ are $\sigma_1,\sigma_2,\sigma_4,\sigma_6,\ldots,\sigma_n$. 

We can now finally determine the left-to-right maxima of $\pi$. It follows from Claim 6 and the definition of $s$ that $\pi=s(\sigma)=\sigma_1\sigma_3\sigma_2\sigma_5\sigma_4
\sigma_7\sigma_6\cdots\sigma_{n-1}\sigma_{n-2}\sigma_n$. Claim 6 also implies that the entries $\sigma_1,\sigma_2,\sigma_4,\sigma_6,\ldots,\sigma_n$, which are the same as the entries $\pi_1,\pi_3,\pi_5,\ldots,\pi_{n-1},\pi_n$, are left-to-right maxima of $\pi$. For example, to see that $\sigma_4$ is a left-to-right maximum of $\pi$, we need to check that $\sigma_4$ is larger than the entries $\sigma_1,\sigma_2,\sigma_3,\sigma_5$. We know that $\sigma_4$ is larger than $\sigma_1,\sigma_2,\sigma_3$ because $\sigma_4$ is a left-to-right maximum of $\sigma$, and it is larger than $\sigma_5$ because $4$ is a descent of $\sigma$. No descent bottom of $\pi$ can be a left-to-right maximum of $\pi$. Since $\des(\pi)=\frac{n-2}{2}$, there are at most $n-\frac{n-2}{2}=\frac{n+2}{2}$ left-to-right maxima of $\pi$. Hence, $\pi_1,\pi_3,\pi_5,\ldots,\pi_{n-1},\pi_n$ are the \emph{only} left-to-right maxima of $\pi$. 
\end{proof}

\begin{proof}[Proof of Theorem \ref{Thm3}]
Proposition \ref{Prop1} completes the proof of the characterization in Theorem \ref{Thm3} when $t=2$. Let us now assume that $n\geq t\geq 3$ and that $n$ and $t$ have the same parity. We already saw at the beginning of this section that every permutation $\pi=\pi_1\cdots\pi_n\in S_n$ whose left-to-right maxima are $\pi_1,\pi_3,\pi_5,\ldots,\pi_{n-t+1},\pi_{n-t+2},\pi_{n-t+3}\ldots,\pi_n$ is $t$-sorted and has $\frac{n-t}{2}$ descents; we need to prove the converse. 

Let $\pi=\pi_1\cdots\pi_n\in S_n$ be an arbitrary $t$-sorted permutation with $\frac{n-t}{2}$ descents. Because $\pi$ is $t$-sorted, it is certainly $(t-2)$-sorted. This means that it ends in the entries $n-t+3,n-t+4,\ldots,n$, so we can write $\pi=\pi'(n-t+3)(n-t+4)\cdots n$ for some $\pi'\in S_{n-t+2}$. There is a $(t-2)$-sorted permutation $\lambda\in S_n$ such that $s^2(\lambda)=\pi$. Since $\lambda$ is $(t-2)$-sorted, we can write $\lambda=\lambda'(n-t+3)(n-t+4)\cdots n$ for some $\lambda'\in S_{n-t+2}$. We now have \[\pi'(n-t+3)(n-t+4)\cdots n=s^2(\lambda)=s^2(\lambda'(n-t+3)(n-t+4)\cdots n)=s^2(\lambda')(n-t+3)(n-t+4)\cdots n,\] so $\pi'=s^2(\lambda')$. This shows that $\pi'$ is a $2$-sorted permutation in $S_{n-t+2}$ with $\frac{n-t}{2}$ descents. According to Proposition \ref{Prop1}, the left-to-right maxima of $\pi'$ are $\pi_1,\pi_3,\pi_5,\ldots,\pi_{n-t+1},\pi_{n-t+2}$. This proves that the left-to-right maxima of $\pi$ are $\pi_1,\pi_3,\pi_5,\ldots,\pi_{n-t+1},\pi_{n-t+2},\pi_{n-t+3},\ldots,\pi_n$, as desired. 

To finally complete the proof of Theorem \ref{Thm3}, we need to show that the number of permutations $\pi=\pi_1\cdots\pi_n\in S_n$ whose left-to-right maxima are $\pi_1,\pi_3,\pi_5,\ldots,\pi_{n-t+1},\pi_{n-t+2},\pi_{n-t+3},\ldots,\pi_n$ is $(n-t-1)!!$ (assuming again that $n\geq t\geq 2$ and that $n\equiv t\pmod 2$). It is convenient to think of constructing the plot of such a permutation, which we can imagine is just a collection of $n$ points in the plane such that no two points lie on a common vertical or horizontal line (without regarding the specific coordinates of these points). We will build the plot by placing points one at a time from left to right. We first place the first point, which will represent the first entry in the permutation. There is $1$ choice for the height of the second point relative to the first point. Namely, the second point must be lower than the first. The third point must be higher than both the first and second points. The fourth point must be lower than the third, but we can freely choose its height relative to the first two points. Thus, there are $3$ choices for the height of the fourth point relative to the first three. The fifth point must be higher than all of the first four points. The sixth point must be lower than the fifth, but we can freely choose its height relative to the first four points. Thus, there are $5$ choices for the height of the sixth point relative to the first five. Continuing in this manner, we find that there are $(n-t-1)!!$ ways to choose the relative heights of the first $n-t$ points. The final $t$ points must be higher than all of the first $n-t$ points, and their heights must be increasing from left to right. Therefore, the total number of ways to construct the plot of $\pi$ is $(n-t-1)!!$. 
\end{proof}

\section{Future Work}
We have given a characterization of the $t$-sorted permutations in $S_n$ that have the maximum possible number of descents when $n\geq t\geq 2$ and $n\equiv t\pmod 2$. A natural next step would involve trying to understand these extremal permutations when $n\not\equiv t\pmod 2$. For example, when $t=2$ and $n$ is odd, we would like to understand (or even just count) the $2$-sorted permutations in $S_n$ with $\frac{n-3}{2}$ descents. This appears to be much more complicated than the case in which $n$ and $t$ have the same parity; any significant progress would be very interesting. 

To elaborate further upon this point, let us note that our proofs often relied upon the characterization given in Theorem \ref{Thm1}. One might hope for a similar characterization that would apply to sorted permutations in $S_n$ with $\frac{n-2}{2}$ descents. One can show that every permutation in $S_n$ with exactly $2$ preimages under $s$ must have exactly $\frac{n-2}{2}$ descents. Unfortunately, the converse is false. The permutation $2134\in S_4$ has $\frac{4-2}{2}=1$ descent, but we can see in Figure \ref{Fig1} that is has $4$ preimages under the stack-sorting map. This lack of an analogue of Theorem \ref{Thm1} is one reason why we might expect studying extremal permutations to be difficult when $n\not\equiv t\pmod 2$. 

\section{Acknowledgments}
The author thanks Mikl\'os B\'ona for suggesting the definition of a $t$-sorted permutation and for suggesting the problem of determining the maximum number of descents that a $t$-sorted permutation of length $n$ can have. He also thanks the anonymous referees for valuable suggestions that improved the presentation of this article. 
The author was supported by a Fannie and John Hertz Foundation Fellowship and an NSF Graduate Research Fellowship.


\begin{thebibliography}{9}
\bibitem{Banderier}
C. Banderier and M. Wallner, The kernel method for lattice paths below a rational slope. In: \emph{Lattice paths combinatorics and applications}, Developments in Mathematics Series, Springer (2018), 1--36.

\bibitem{Bona}
M. B\'ona, Combinatorics of permutations. CRC Press, 2012. 

\bibitem{BonaSimplicial}
M. B\'ona, A simplicial complex of 2-stack sortable permutations. \emph{Adv. Appl. Math.}, {\bf 29} (2002), 499--508.

\bibitem{BonaWords}
M. B\'ona, Stack words and a bound for $3$-stack sortable permutations. arXiv:1903.04113.

\bibitem{BonaSurvey}
M. B\'ona, A survey of stack-sorting disciplines. \emph{Electron. J. Combin.}, {\bf 9.2} (2003): \# 16.

\bibitem{BonaSymmetry}
M. B\'ona, Symmetry and unimodality in $t$-stack sortable permutations. \emph{J. Combin. Theory Ser.
A}, {\bf 98.1} (2002), 201–-209.

\bibitem{Bousquet3}
M. Bousquet-M\'elou and A. Jehanne, Polynomial equations with one catalytic variable, algebraic series and map enumeration. \emph{J. Combin. Theory Ser.
B}, {\bf 96} (2006), 623--672.

\bibitem{Bousquet98}
M. Bousquet-M\'elou, Multi-statistic enumeration of two-stack sortable permutations. \emph{Electron. J. Combin.}, {\bf 5} (1998), \#R21.

\bibitem{Bousquet}
M. Bousquet-M\'elou, Sorted and/or sortable permutations. \emph{Discrete Math.}, {\bf 225} (2000), 25--50.

\bibitem{Bouvel}
M. Bouvel and O. Guibert, Refined enumeration of permutations sorted with two stacks and a $D_8$-symmetry. \emph{Ann. Comb.}, {\bf 18} (2014), 199--232. 

\bibitem{BrandenActions}
P. Br\"and\'en, Actions on permutations and unimodality of descent polynomials. \emph{European J. Combin.}, {\bf 29} (2008), 514--531. 

\bibitem{Branden3}
P. Br\"and\'en, On linear transformations preserving the P\'olya frequency property. \emph{Trans. Amer. Math. Soc.},  {\bf 358} (2006), 3697--3716.

\bibitem{Claesson}
A. Claesson and H. \'Ulfarsson, Sorting and preimages of pattern classes. arXiv:1203.2437. 

\bibitem{Cori}
R. Cori, B. Jacquard, and G. Schaeffer, Description trees for some families of planar maps. \emph{Proceedings of the 9th FPSAC}, (1997). 

\bibitem{DefantCatalan}
C. Defant, Catalan intervals and uniquely sorted permutations. arXiv:1904.02627.

\bibitem{DefantCounting}
C. Defant, Counting $3$-stack-sortable permutations. arXiv:1903.09138.  

\bibitem{DefantEnumeration}
C. Defant, Enumeration of stack-sorting preimages via a decomposition lemma. arXiv:1904.02829.

\bibitem{DefantFertility}
C. Defant, Fertility numbers. arXiv:1809.04421. 

\bibitem{DefantFertilityWilf}
C. Defant, Fertility, strong fertility, and postorder Wilf equivalence. arXiv:1904.03115.

\bibitem{DefantPolyurethane}
C. Defant, Polyurethane toggles. arXiv:1904.06283.

\bibitem{DefantPostorder}
C. Defant, Postorder preimages. \emph{Discrete Math. Theor. Comput. Sci.}, {\bf 19} (2017). 

\bibitem{DefantPreimages}
C. Defant, Preimages under the stack-sorting algorithm. \emph{Graphs Combin.}, {\bf 33} (2017), 103--122. 

\bibitem{DefantClass}
C. Defant, Stack-sorting preimages of permutation classes. arXiv:1809.03123. 

\bibitem{DefantEngenMiller}
C. Defant, M. Engen, and J. A. Miller, Stack-sorting, set partitions, and Lassalle's sequence. arXiv:1809.01340.  

\bibitem{DefantKravitz}
C. Defant and N. Kravitz, Stack-sorting for words. arXiv:1809.09158.  

\bibitem{Dulucq}
S. Dulucq, S. Gire, and O. Guibert, A combinatorial proof of J. West's conjecture. \emph{Discrete Math.}, {\bf 187} (1998), 71--96.

\bibitem{Dulucq2}
S. Dulucq, S. Gire, and J. West, Permutations with forbidden subsequences and nonseparable planar maps. \emph{Discrete Math.}, {\bf 153.1} (1996), 85--103.

\bibitem{Fang}
W. Fang, Fighting fish and two-stack-sortable permutations. \emph{S\'em. Lothar. Combin.}, {\bf 80} (2018).

\bibitem{Goulden}
I. Goulden and J. West, Raney paths and a combinatorial relationship between rooted nonseparable planar maps and two-stack-sortable permutations. \emph{J. Combin. Theory Ser. A.}, {\bf 75.2} (1996), 220--242.

\bibitem{Josuat}
M. Josuat-Verg\`es, Cumulants of the $q$-semicircular law, Tutte polynomials, and heaps. \emph{Canad. J. Math.}, {\bf 65} (2013), 863--878. 

\bibitem{Kitaev}
S. Kitaev, Patterns in Permutations and Words. Monographs in Theoretical Computer
Science. Springer, Heidelberg, 2011.

\bibitem{Knuth}
D. E. Knuth, The Art of Computer Programming, volume 1, Fundamental Algorithms.
Addison-Wesley, Reading, Massachusetts, 1973.

\bibitem{Lassalle}
M. Lassalle, Two integer sequences related to Catalan numbers. \emph{J. Combin. Theory Ser. A}, {\bf 119} (2012), 923--935. 

\bibitem{OEIS}
The On-Line Encyclopedia of Integer Sequences, published electronically at http://oeis.org, 2019.

\bibitem{Ulfarsson}
H. \'Ulfarsson, Describing West-$3$-stack-sortable permutations with permutation patterns. \emph{S\'em. Lothar. Combin.}, {\bf 67} (2012).

\bibitem{West}
J. West, Permutations with restricted subsequences and stack-sortable permutations, Ph.D. Thesis, MIT, 1990.

\bibitem{Zeilberger}
D. Zeilberger, A proof of Julian West's conjecture that the number of two-stack-sortable permutations of length
$n$ is $2(3n)!/((n + 1)!(2n + 1)!)$. \emph{Discrete Math.}, {\bf 102} (1992), 85--93.  
\end{thebibliography}
\end{document}